\documentclass[12pt]{amsart}

\usepackage[english]{babel}
\usepackage[utf8]{inputenc}

\usepackage{amsmath,amsfonts,amssymb,amscd,amsthm}
\usepackage[linktocpage,unicode]{hyperref}

\usepackage{latexsym,graphics,esdiff}
\usepackage[dvips]{color}
\usepackage[dvips]{graphicx}
\usepackage{pstricks}
\usepackage[left=3cm,right=3cm,top=3.4cm,bottom=3.4cm]{geometry}

\title[Derivatives of the tree function]{Derivatives the tree function}
\subjclass[2000]{05C05, 11B83, 26A48}
\keywords{Cayley trees, Bernstein function, completely monotonic functions}
\date{}
\author{Matthieu Josuat-Vergès}
\thanks{This work was supported by ANR project CARMA}
\address{Institut Gaspard Monge, Université Paris-Est Marne-la-Vallée\\
5 Bd. Descartes\\
Champs-sur-Marne\\
77454 Marne-la-Vallée cedex 2\\ France}
\email{matthieu.josuat-verges@univ-mlv.fr}

\hypersetup{
    unicode=false,          
    pdftoolbar=true,        
    pdfmenubar=true,        
    pdffitwindow=false,     
    pdfstartview={FitH},    
    pdftitle={Derivatives of the tree function},
    pdfauthor={Matthieu Josuat-Vergès},
    pdfsubject={},
    pdfkeywords={},
    pdfnewwindow=true,      
    colorlinks=true,       
    linkcolor=red,          
    citecolor=blue,        
    filecolor=magenta,      
    urlcolor=cyan           
}

\newtheorem{theo}{Theorem}[section]

\newtheorem{lemm}[theo]{Lemma}
\newtheorem{prop}[theo]{Proposition}

\theoremstyle{definition}
\newtheorem{defi}[theo]{Definition}

\newcommand{\vertex}[3]{
  \pscircle[fillstyle=solid,fillcolor=white](#1,#2){0.25}
  \cput[linestyle=none](#1,#2){\small #3}
}

\newcommand{\rvertex}[3]{
  \pscircle[doubleline=true,fillstyle=solid,fillcolor=white](#1,#2){0.25}
  \cput[linestyle=none](#1,#2){\small #3}
}

\DeclareMathOperator{\unl}{unl}
\DeclareMathOperator{\imp}{imp}

\begin{document}

\begin{abstract}
We study some sequences of polynomials that appear when we consider the successive derivatives 
of the tree function (or Lambert's W function). We show in particular that they are related with
a generalization of Cayley trees, called Greg trees.
Besides the combinatorial result in itself, it is interesting to see how this is related
with previous work: similar problems were considered first by Ramanujan, and more recently in the
theory of completely monotonic functions and its link with probability.
Also of great interest is the fact that these Greg trees were introduced in a problem of textual 
criticism, as a kind a genealogical trees, where they had a priori no mathematical meaning.
\end{abstract}

\maketitle


\section{Introduction}

Lambert's W function is a special function that has been widely studied and naturally appear in 
various contexts in pure or applied sciences, see for example the comprehensive article
by Corless et al. \cite{corless}.
It is usually defined by the equation
\[
  W(z)e^{W(z)} = z.
\]
We are interested here in the computation of derivatives of this function.
We can derive the first values from the defining equation, and the general pattern that
appears is 
\begin{equation}\label{defP}
  \diff[n]{}{z} W(z) =  \frac{ \exp{(-nW(z))}  }{  (1+W(z))^{2n-1}  } P_n(W(z))
\end{equation}
for some polynomials $P_n$ (this is done in \cite{corless}). 
Kalugin and Jeffrey \cite{kalugin} proved that the polynomials $(-1)^{n-1}P_n$ have positive coefficients,
settling a conjecture of Sokal. This result was motivated by the fact
that we can deduce the sign of all these functions: they are alternately positive and negative,
and it follows that $W(z)$ is a so-called Bernstein function \cite{schilling}.
Various similar results have been obtained in a rather analytical perspective \cite{kalugin2,kalugin3},
and a probability law related with $W(z)$ as a Bernstein function is studied in \cite{pakes}.

In fact, the polynomials $P_n$ are closely related with another sequence that appeared first 
in Ramanujan's notebook (see \cite{berndt} and \cite{ramanujan}), and inspired quite a few
combinatorial works \cite{chen,dumont,guo,howard,linzeng,zeng}. The goal of this article 
is to introduce {\it Greg trees} in this context. They are a generalization of Cayley trees, and 
were defined by Flight \cite{flight} as some kind of genealogical trees for manuscripts.
(Flight named these trees after Sir Walton Wilson Greg, a renowned British scholar who worked
in textual criticism, and was in particular an expert of Shakespeare's texts \cite{wilson}.)
Similar trees were also considered in \cite{felsenstein} as some kind of phylogenetic trees.
The connection was noticed by 
Knuth, in the correction of Exercise~50 of Section~7.2.1.5 in the Prefascicle~3B of the Art of Computer Programming
\cite{knuth}. Here, we explain this connection, by showing that some operations on Greg trees can prove
formulas of the kind of Equation~\eqref{defP}.
In fact, we study here three closely-related sequences of polynomials $F_n$, $G_n$ and $H_n$.
The first one is related with Ramanujan's polynomials, and the latter two give the enumeration of 
Greg trees, respectively in the rooted and unrooted cases. (Note that these sequences can be seen 
as particular cases of a two-variable polynomials, see Section~\ref{fr}. However, from the point
of view of the enumeration of Greg trees, only these cases are consedered.)

\section{Definitions}

Since the defining equation $W(z)e^{W(z)}=z$ has several solutions, $W(z)$ can be considered as a multivalued function.
Here we only consider the principal branch, and see $W(z)$ as a holomorphic function on $\mathbb{C}-(-\infty,-\frac 1e]$,
such that
\[
  W(z) = \sum_{n \geq 1 }   (-n)^{n-1} \frac{z^n}{n!}
\]
in a neighborhood of $0$ (see \cite{corless}).
From the function $W(z)$, we define the tree function:
\[  
  T_1(z) = -W(-z).
\]
It is well-known that $n^{n-1}$ is number of rooted Cayley trees \cite{cayley}, so that $T_1(z)$ is the
exponential generating function of these.
More generally, let us define:
\[
  T_\alpha(z) = \sum_{n \geq 1 }   n^{n-\alpha} \frac{z^n}{n!}. 
\]
We have $T(z)=T_1(z)$ and we will also consider the cases where $\alpha=0$ or $\alpha=2$.
In particular, the generating function unrooted Cayley trees is:
\[
  T_2(z) = \sum_{n \geq 1 }   n^{n-2} \frac{z^n}{n!}.  
\]
There are several relations linking these functions, for example:
\[
  T_0(z)=\frac{1}{1-T(z)}, \qquad T_2'(z) = \frac{T(z)}z, \qquad T_2(z)=T(z) - \tfrac 12 T(z)^2.
\]
We will use these identities in the sequel but leave as an exercise to find either an analytical
or a bijective proof for each of them. Whereas $T_1$ and $T_2$ are clearly interesting in a
combinatorial context, we will see that $T_0$ is also related with a nice sequence of polynomials,
but we do not get into other cases.

\begin{prop}
There are polynomials $F_n(x)$, $G_n(x)$ and $H_n(x)$ such that for any $n\geq1$ we have:
\begin{equation}\label{defF}
\diff[n]{}{z} T_0(z) =  \frac{  \exp{ (nT(z)) }  }{  (1-T(z))^{n+2} }  F_n\left(   \frac{T(z)}{1-T(z)}   \right).
\end{equation}

\begin{equation}\label{defG}
\diff[n]{}{z} T_1(z) =  \frac{  \exp{ (nT(z)) }  }{  (1-T(z))^n }  G_n\left(   \frac{T(z)}{1-T(z)}   \right).
\end{equation}

\begin{equation}\label{defH}
\diff[n]{}{z} T_2(z) =  \frac{  \exp{ (nT(z)) }  }{  (1-T(z))^{n-1} } H_n\left(  \frac{T(z)}{1-T(z)}   \right).
\end{equation}

With the initial case $F_1(x)=G_1(x)=H_1(x)=1$, these polynomials satisfy the recursion:
\begin{equation} \label{recGH} \begin{split}
F_{n+1}(x) &=  (2n+2 + (n+2)x)  F_n(x) + (1+x)^2 F'_n(x), \\
G_{n+1}(x) &=  (2n + nx)    G_n(x) + (1+x)^2 G'_n(x), \\
H_{n+1}(x) &=  (2n-1 + (n-1)x) H_n(x) + (1+x)^2 H'_n(x).
\end{split} \end{equation}
\end{prop}

\begin{proof}
We only present the case of $T_1(z)$ and the other two are similar.
By differentiation of $T_1(z) = z e^{T_1(z)}$, we get:
\[
  T_1'(z) = e^{T_1(z)} + zT_1'(z) e^{T_1(z)},
\]
hence:
\[
  T_1'(z) = \frac{e^{T_1(z)}}{1-ze^{T_1(z)}} = \frac{e^{T_1(z)}}{1-T_1(z)}.
\]
This proves the case $n=1$, and the general case is done inductively.
Indeed, suppose that $G_n$ exists such that Identity $\eqref{defG}$ is true,
by differentiating both sides, we get:
\begin{align*}
\diff[n+1]{}{z} T_1(z) & =
    \tfrac{ nT'(z) e^{ nT(z) }  }{  (1-T(z))^n }  G_n\left(   \tfrac{T(z)}{1-T(z)}   \right) \\
  & \qquad + \tfrac{ nT'(z) e^{ nT(z) }  }{  (1-T(z))^{n+1} }  G_n\left(   \tfrac{T(z)}{1-T(z)}   \right)
  + \tfrac{ T'(z)  e^{ nT(z) }  }{  (1-T(z))^{n+2} }  G_n'\left(   \tfrac{T(z)}{1-T(z)}   \right),
\end{align*}
and then:
\begin{align*}
\diff[n+1]{}{z} T_1(z) & = 
    \tfrac{ n e^{ (n+1)T(z) }  }{  (1-T(z))^{n+1} }  G_n\left(   \tfrac{T(z)}{1-T(z)}   \right) \\
  & \qquad + \tfrac{ n e^{ (n+1)T(z) }  }{  (1-T(z))^{n+2} }  G_n\left(   \tfrac{T(z)}{1-T(z)}   \right)
  + \tfrac{ e^{ (n+1)T(z) }  }{  (1-T(z))^{n+3} }  G_n'\left(   \tfrac{T(z)}{1-T(z)}   \right),
\end{align*}
\begin{align*}
\diff[n+1]{}{z} T_1(z) = \tfrac{ e^{ (n+1)T(z) }  }{  (1-T(z))^{n+1} }
                         \left( n \tfrac{2-T(z)}{1-T(z)} G_n\left(   \tfrac{T(z)}{1-T(z)}   \right)  
                         + \tfrac{1}{(1-T(z))^2}  G_n'\left(   \tfrac{T(z)}{1-T(z)} \right) \right).
\end{align*}
This equation shows that $G_{n+1}$ exists, with:
\begin{align*}
 G_{n+1}\left(   \tfrac{T(z)}{1-T(z)}   \right) = n \tfrac{2-T(z)}{1-T(z)} G_n\left(   \tfrac{T(z)}{1-T(z)}   \right)  +  \tfrac{1}{(1-T(z))^2}  G_n'\left(   \tfrac{T(z)}{1-T(z)} \right).
\end{align*}
This gives the recursion since we have $G_{n+1}$ in terms of $G_n$.

Similar computations give the result for $F_n$ and $H_n$.
\end{proof}

The first values of these polynomials are as follows:
\[ 
\begingroup
\everymath{\scriptstyle}
\begin{array}{ c|c|c }
n &  F_n(x) & G_n(x)  \\
\hline 
1 & 1 & 1  \\
2 & 3x+4 & x+2  \\
3 & 15x^2 + 40 x + 27 & 3x^2+10x+9  \\
4 & 105 x^3 + 420 x^2 + 565 x + 256 & 15x^3+70x^2+113x+64  \\
5 & 945 x^4  + 5040 x^3  + 10150 x^2  + 9156 x + 3125 & 105x^4+630x^3+1450x^2+1526x+625 \\
6 &  10395x^5+69300x^4+185850x^3+250768x^2+170359x+46656   & 945x^5+6930x^4+20650x^3+31346x^2+24337x+7776
\end{array} 
\endgroup
\]

\[ 
\begingroup
\everymath{\scriptstyle}
\begin{array}{ c|c }
n  & H_n(x) \\
\hline 
1 &  1 \\
2 &  1 \\
3 &  x+3 \\
4 &  3x^2+13x+16 \\
5 &  15x^3+85x^2+171x+125 \\
6 &  105x^4+735x^3+2005x^2+2551x+1296 \\
7 &  945x^5+7875x^4+26950x^3+47586x^2+43653x+16807
\end{array}
\endgroup
\]

The coefficients of $G_n(x)$ are Sloane's \href{http://oeis.org/A048160}{A048160}, and
those of $H_n(x)$ are \href{http://oeis.org/A048159}{A048159}.
These two triangle of integers were defined by Flight \cite{flight}. The row sums, i.e. the integers
$G_n(1)$, appear in Felsenstein's article \cite{felsenstein}.

We also have to examine some of the shifted polynomials:
\[ 
\begingroup
\everymath{\scriptstyle}
\begin{array}{ c|c|c }
n &  F_n(x-1) & G_n(x-1)  \\
\hline 
1 & 1 & 1  \\
2 & 3x+1 & x+1  \\
3 & 15x^2 + 10 x + 2 & 3x^2+4x+2  \\
4 & 105 x^3 + 105 x^2 + 40 x + 6 & 15x^3+25x^2+18x+6  \\
5 & 945 x^4  + 1260 x^3  + 700 x^2  + 196 x + 24 & 105x^4+210x^3+190x^2+96x+24 \\
6 & 10395x^5+17325x^4+12600x^3+5068x^2+1148x+120 & 945x^5+2205x^4+2380x^3+1526x^2+600x+120
\end{array}
\endgroup
\]

\[ 
\begingroup
\everymath{\scriptstyle}
\begin{array}{ c|c }
n  & H_n(x-1) \\
\hline 
1 &  1 \\
2 &  1 \\
3 &  x+2 \\
4 &  3x^2+7x+6 \\
5 &  15x^3+40x^2+46x+24 \\
6 &  105x^4+315x^3+430x^2+326x+120 \\
7 & 945x^5+3150x^4+4900x^3+4536x^2+2556x+720
\end{array}
\endgroup
\]

The coefficients of $F_n(x-1)$ are \href{http://oeis.org/A075856}{A075856} and appeared in Ramanujan's notebooks \cite{ramanujan}
(with a definition equivalent to the present one), and several other works \cite{gould,goulden}.

$G_n(x-1)$ appeared in \cite{shor} where it is shown that these polynomials 
count rooted Cayley trees according to the number or {\it improper edges} (see below),
the unrooted analog $H_n(x-1)$ appeared in Zeng's article \cite{zeng}.
The coefficients of these polynomial are respectively  \href{http://oeis.org/A054589}{A054589}
and  \href{http://oeis.org/A217922}{A217922}.

\section{Greg trees}

\begin{defi}[\cite{flight}]
Let $n\geq 0$.
A {\it Greg tree} of size $n$ is a tree such that:
\begin{itemize}
 \item there are $n$ vertices labeled by integers from $1$ to $n$, and the other
         vertices are unlabeled,
\item the unlabeled vertices have degree at least 3.
\end{itemize}
Let $\mathcal{G}_n$ denote the set of {\it Greg trees} with $n$ labeled vertices.
Let $\unl(T)$ denote the number of unlabeled vertices of some $T\in\mathcal{G}_n$.
Let $\mathcal{C}_n\subset\mathcal{G}_n$ denote the subset of Cayley trees, i.e. those $T$ with $\unl(T)=0$.
\end{defi}

Figure~\ref{greg3} shows the Greg trees of size 3.

Since we do not specify a bound on the number of unlabeled vertices, it is not {\it a priori} clear that there 
is a finite number of Greg trees of size $n$. But the condition that unlabeled vertices have degree at least 3 
gives such a bound. Indeed, let $u$ be the number of unlabeled vertices in a Greg tree of size $n$,
so $u+n$ is the total number of vertices and there are $u+n+1$ edges. 
The condition on the degree easily gives $2(u+n+1)\geq 3u$, hence $u\leq 2(n+1)$.

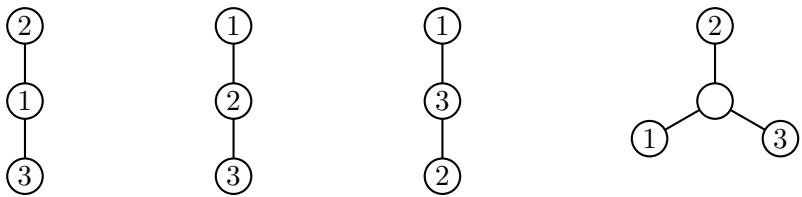
\begin{figure}[h!tp]
\begin{pspicture}(-0.25,-0.25)(0.25,2.25)
\psline(0,0)(0,2) 
\vertex{0}{2}{2}
\vertex{0}{1}{1}
\vertex{0}{0}{3}
\end{pspicture}
\hspace{2cm}
\begin{pspicture}(-0.25,-0.25)(0.25,2.25)
\psline(0,0)(0,2) 
\vertex{0}{2}{1}
\vertex{0}{1}{2}
\vertex{0}{0}{3}
\end{pspicture}
\hspace{2cm}
\begin{pspicture}(-0.25,-0.25)(0.25,2.25)
\psline(0,0)(0,2) 
\vertex{0}{2}{1}
\vertex{0}{1}{3}
\vertex{0}{0}{2}
\end{pspicture}
\hspace{2cm}
\begin{pspicture}(-1.1,-0.25)(1.1,2.25)
\psline(0,1)(0,2)
\psline(0,1)(-0.87,0.5)
\psline(0,1)(0.87,0.5)
\vertex{0}{2}{2}
\vertex{-0.87}{0.5}{1}
\vertex{0.87}{0.5}{3}
\vertex{0}{1}{}
\end{pspicture}
\caption{Greg trees of size 3. \label{greg3}}
\end{figure}

\begin{defi}
Let $\mathcal{G}^\bullet_n$
denote the set of {\it rooted Greg trees}, i.e. Greg trees of size $n$ with a distinguished vertex called the {\it root},
but with the additional rule that if the root is unlabeled it may have degree 2.

Let $\mathcal{C}^\bullet_n\subset\mathcal{G}^\bullet_n$ denote the subset of rooted Cayley trees.
\end{defi}

See Figure~\ref{rgreg3} for the rooted Greg trees of size 2. We represent the root as a vertex with a double line circle.
Note that since only the root is allowed to have degree 2, the same argument as in the unrooted case gives
a bound on the number of unlabeled vertices, so that there is a finite number of rooted Greg trees of size $n$.

\begin{figure}[h!tp]
\begin{pspicture}(0.25,0.25)(0.75,1.75)
 \psline(0.5,0.5)(0.5,1.5)
 \rvertex{0.5}{1.5}{1}
 \vertex{0.5}{0.5}{2}
\end{pspicture}
\hspace{2cm}
\begin{pspicture}(0.25,0.25)(0.75,1.75)
 \psline(0.5,0.5)(0.5,1.5)
 \vertex{0.5}{1.5}{1}
 \rvertex{0.5}{0.5}{2} 
\end{pspicture}
\hspace{2cm}
\begin{pspicture}(0.25,0)(0.75,1.5)
 \psline(0.5,0.5)(1.25,1)(2,0.5)
 \vertex{0.5}{0.5}{1}
 \vertex{2}{0.5}{2}
 \rvertex{1.25}{1}{}
\end{pspicture}
 \caption{Rooted Greg trees of size 2. \label{rgreg3}}
\end{figure}
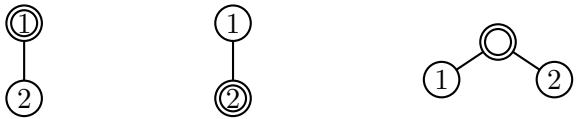

\begin{theo}
 For any $n\geq1$, we have:
\begin{equation}
 G_n(x) = \sum_{T\in \mathcal{G}^\bullet_n} x^{\unl(T)}, \qquad H_n(x) = \sum_{T\in \mathcal{G}_n} x^{\unl(T)}.
\end{equation}
\end{theo}

It possible to prove this theorem by examining how to build a Greg tree of size $n+1$ from a Greg tree of size
$n$, which leads precisely to the relations in \eqref{recGH} as was done in \cite{flight}. But we take a different
point of view here. Rather than the recursion in \eqref{recGH}, we show that the structure of the right-hand
sides of \eqref{defG} and \eqref{defH} is explained through some combinatorial operations on Greg trees.

\begin{defi}
Let $T\in\mathcal{C}_m$ and $n<m$. We define a Greg tree $ T|_n \in \mathcal{G}_{n}$ as the result of the
following process:
\begin{list}{\labelitemi}{\leftmargin=1em}
 \item Remove all labels greater than $n$ (thus creating new unlabeled vertices).
 \item Remove each unlabeled vertex of degree 2, and join the two pending edges into a single edge.
 \item Remove each unlabeled vertex of degree 1, as well as the incident edge.
 \item Repeat the last two steps, until the result is a Greg tree.
\end{list}
\end{defi}

This process is illustrated in Figure~\ref{restr}.

\begin{figure}[h!tp] 
\begin{pspicture}(3.1,2.5)
\psline(1,1)(0.4,0.3)
\psline(1,1)(0.3,1.1)
\psline(1,1)(1,1.8)
\psline(1,1.8)(1.6,2.2) 
\psline(1.6,2.2)(2.2,1.8)
\psline(1,1)(1.4,0.4)
\psline(1.4,0.4)(2.1,0.4)
\psline(2.1,0.4)(2.3,1)
\psline(2.1,0.4)(2.9,0.4)
\vertex{0.4}{0.3}{1}
\vertex{1}{1.8}{2}
\vertex{2.9}{0.4}{3}
\vertex{1.6}{2.2}{4}
\vertex{1.4}{0.4}{5}
\vertex{0.3}{1.1}{6}
\vertex{1}{1}{7}
\vertex{2.1}{0.4}{8}
\vertex{2.2}{1.8}{9}
\vertex{2.3}{1}{10}
\end{pspicture}
\hspace{1cm}
\begin{pspicture}(2,2) 
\psline(1,1)(0.4,0.3)
\psline(1,1)(1,1.8)
\psline(1,1.8)(1.6,2.2) 
\psline(1,1)(1.9,0.4)
\vertex{0.4}{0.3}{1}
\vertex{1}{1.8}{2}
\vertex{1.9}{0.4}{3}
\vertex{1.6}{2.2}{4}
\vertex{1}{1}{}
\end{pspicture}
 \caption{ The operation $T\mapsto T|_n$ (here $n=4$). \label{restr}}
\end{figure}
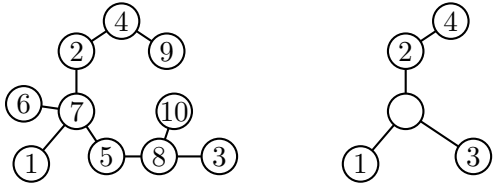

\begin{prop} \label{propgen}
Let $T\in\mathcal{G}_n$. We have:
\begin{equation}
 \sum_{m \geq n} \sum_{\substack{ X \in \mathcal{C}_m \\  X|_n = T    } } \frac{z^{m-n}}{(m-n)!} = 
 \frac{e^{nT(z)}}{(1-T(z))^{n-1}} \bigg( \frac{T(z)}{1-T(z)} \bigg)^{\unl(T)}.
\end{equation}
\end{prop}

\begin{proof}
In this proof, we make an extensive use of the symbolic method for labeled combinatorial objects \cite{flajolet}, 
which permits to obtain the result in a purely combinatorial manner.
The idea is to find a canonical decomposition for a Cayley tree $X \in \mathcal{C}_m$ satisfying $X|_n=T$, as a 
collection of rooted Cayley trees indexed in a particular way. Note that this tree $X$ is considered as 
having size $m-n$, i.e. only the labels greater than $n$ are part of the labeled combinatorial structure.
We rewrite the right hand side of the equation as:
\begin{equation} \label{rewr}
  e^{nT(z)}   \bigg( \frac{1}{1-T(z)} \bigg)^{n-1+\unl(T)}  T(z)^{\unl(T)}.
\end{equation}
Then the construction is as follows:
\begin{itemize}
 \item To each of the $n$ labeled vertices of $T$, we associate an unordered collection of rooted trees,
       each giving a factor $e^{T(z)}$.
 \item To each of the $n-1+\unl(T)$ edges of $T$, we associate an ordered collection of rooted trees, 
       each giving a factor $\frac{1}{1-T(z)}$.
 \item To each of the $\unl(T)$ unlabeled vertices of $T$, we associate a rooted tree,
       each giving a factor $T(z)$.
\end{itemize}
Here, we understand that all these trees are part of the same combinatorial structure, i.e. each integer
between $n+1$ and $m$ appears exactly once as the label of a vertex of some tree.
We have thus defined a combinatorial class whose generating function is clearly 
given by Equation~\eqref{rewr}, and it remains to check that it maps bijectively to
the set of $X$ such that $X|_n=T$.
We can build such a Cayley tree $X$ satisfying $X|_n=T$ from a collection of trees as above in a straightforward way:
\begin{itemize}
 \item An unordered collection $C$ of rooted trees associated to a labeled vertex $v$ is attached to $T$
       by adding an edge from $v$ to the root of each tree in $C$.
 \item An ordered collection $C$ of rooted trees associated to an edge $e$ (say, from vertices $v_1$ to $v_2$) 
       is attached to $T$ by: removing $e$, adding an edge from $v_1$ to the root of the first tree in $C$, an
       edge from the root of a tree in $C$ to the root of the next tree, and an edge from the root 
       of the last tree in $C$ to $v_2$.
 \item A rooted tree associated with an unlabeled vertex of $T$ is attached to $T$ by putting its
       root in place of the unlabeled vertex.
\end{itemize}
The fact that we have a bijection can be proved formally by an induction on $m$, for example. We omit details.
\end{proof}

As a consequence of the previous proposition, we have:
\begin{align*}
  \diff[n]{}{z} T_2(z) & = \diff[n]{}{z} \sum_{m\geq0} \sum_{T\in\mathcal{C}_m} \frac{z^m}{m!}
                         = \sum_{ m \geq n } \sum_{T \in \mathcal{C}_m } \frac{z^{m-n}}{(m-n)!} 
                         = \sum_{U \in\mathcal{G}_n }  \sum_{m \geq n} \sum_{\substack{ T \in \mathcal{C}_m \\  T|_n = U    } } \frac{z^m}{(m-n)!} \\
                       & = \sum_{U \in\mathcal{G}_n }  \frac{e^{nT(z)}}{(1-T(z))^{n-1}} \bigg( \frac{T(z)}{1-T(z)} \bigg)^{\unl(T)}
                         = \frac{e^{nT(z)}}{(1-T(z))^{n-1}} H_n \bigg( \frac{T(z)}{1-T(z)} \bigg).
\end{align*}
So the combinatorics of Greg trees proves the formula in Equation~\eqref{defH}. Note in particular that
we do not use the recursion for the polynomial $H_n(x)$ or for the construction of Greg trees.

As for the rooted case, the operation $T\mapsto T|_n$ can be defined in a similar way. The only difference is that 
an unlabeled vertex with degree 2 is removed only if it is not the root.
Similarly, we have:

\begin{prop}
Let $T\in\mathcal{G}^\bullet_n$. We have:
\begin{equation}
 \sum_{m \geq n} \sum_{\substack{ X \in \mathcal{C}^\bullet_m \\  X|_n = T    } } \frac{z^{m-n}}{(m-n)!} = 
 \frac{e^{nT(z)}}{(1-T(z))^n} \bigg( \frac{T(z)}{1-T(z)} \bigg)^{\unl(T)}.
\end{equation}
\end{prop}

And we can deduce Equation~\eqref{defG} in the same way as we obtained Equation~\eqref{defH}.

For the sake of completeness, let us mention that we could have considered a slightly more general definition 
for rooted Greg trees, by allowing an unlabeled root to have degree 1 or 2. Then the generating function in 
size $n$ is the polynomial $(1+x)G_n(x)$, and this can be easily proved bijectively (the term $xG_n(x)$ gives
the generating function for the case where the root is unlabeled and has degree 1).
The formulas for $G_n(x)$ can be modified accordingly. In particular, the previous proposition would be 
perhaps more natural with this alternative definition, as there would be $(1-T(z))^{n-1}$ in the denominator
just as in Proposition~\ref{propgen} instead of $(1-T(z))^{n}$.

The reason of our choice is that the present definition 
of Greg trees agrees with \cite{flight}, and considering $G_n(x)$ rather than $(1+x)G_n(x)$ makes a better
connection with previous works mentionned in the introduction.
It would also be possible to define bi-rooted Greg trees, where we have two (ordered) roots allowed to have 
degree 1 or 2, and the generating function of these is $\sum x^{\unl(T)}= (1+x)^3F_n(x)$. The proof is along the
same line: they appear by the operation $T\mapsto T|_n$ defined on bi-rooted Cayley trees (whose exponential 
generating function is $T_0(z)$). 

\section{Generating functions}

It is possible to write down formulas for the generating functions of the polynomials $F_n(x)$, $G_n(x)$
and $H_n(x)$. These essentially follow from Equations~\eqref{defF}, \eqref{defG}, \eqref{defH}
and Taylor series expansions.

\begin{lemm}
The compositional inverse of $\frac{T(z)}{1-T(z)}$ is
$
  \frac{z}{1+z} \exp\big( - \frac{z}{1+z} \big)
$.
\end{lemm}

\begin{proof}
We see $\frac{T(z)}{1-T(z)}$ as the composition of $T(z)$ and $\frac{z}{1-z}$.
The compositional inverse of these two functions are respectively $ze^{-z}$ and
$\frac{z}{1+z}$. The result follows.
\end{proof}

Let us define $F_0(x)=\frac{1}{1+x}$, $G_0(x) = \frac{x}{1+x}$, and $H_0(x) = \frac{x(x+2)}{2(x+1)}$. 
Although these are not polynomials, we can check that Equations~\eqref{defF}, \eqref{defG}, \eqref{defH}, 
and the recursions in Equation~\eqref{recGH} are true for $n=0$, so these definitions are rather natural. 
In particular, the generating functions have a nice form as follows:

\begin{theo}
\begin{align*}
  \sum_{n\geq 0 } \frac{u^n}{n!} F_n(x) &= \frac{1}{(1+x)^2} T_0 \left( \frac{u+x}{1+x}e^{-\frac{x}{1+x}} \right), \\
  \sum_{n\geq0} \frac{u^n}{n!} G_n(x)   &= T\left( \frac{u+x}{1+x}e^{-\frac{x}{1+x}} \right), \\
  \sum_{n\geq0} \frac{u^n}{n!} H_n(x)   &= (1+x) T_2 \left( \frac{u+x}{1+x}e^{-\frac{x}{1+x}} \right).
\end{align*}
\end{theo}

\begin{proof}
Let us begin with the case of $G_n(x)$.
By a Taylor series expansion and the definition of $G_n(x)$ in Equation~\eqref{defG}, we have:
\[
  T(y+z) = \sum_{n\geq0} \frac{y^n}{n!} \diff[n]{}{z} T(z)
         = \sum_{n\geq0} \frac{1}{n!} \left( \frac{e^{T(z)}y}{1-T(z)}  \right)^n G_n\left( \frac{T(z)}{1-T(z)}\right).
\]
Therefore, to get the result we can make the substitions:
\[
  u= \frac{e^{T(z)}y}{1-T(z)}, \qquad x=\frac{T(z)}{1-T(z)}.
\]
From the lemma above, we obtain that $z= \frac{x}{1+x}e^{-\frac{x}{1+x}}$.
Since $T(z)=ze^{T(z)}$, we can check that $xy=zu$. So:
\[
  y+z = \left(\tfrac ux +1\right) z =  \frac{u+x}{1+x}e^{-\frac{x}{1+x}}
\]
and the result follows. The other two generating functions are obtained by expanding
$T_0(y+z)$ and $T_2(y+z)$ and making the same substitution.
\end{proof}

Note that we have proved in particular that
\[
  H(x,u) = (1+x) \left( G(x,u) - \tfrac 12 G(x,u)^2 \right)
\]
where $G(x,u)=\sum_{n\geq 0}\frac{u^n}{n!} G_n(x)$ and $H(x,u)=\sum_{n\geq 0}\frac{u^n}{n!} H_n(x)$.
To understand this bijectively, it is better to use the generating functions without constant terms:
\[
  \tilde G = \sum_{n\geq 1} \frac{u^n}{n!} G_n(x) , \qquad \tilde H = \sum_{n\geq 1} \frac{u^n}{n!} H_n(x) .
\]
Then the relation becomes
$  \tilde H = \tilde G - \frac{1+x}{2} \tilde G ^ 2$,
which we can rewrite $ \tilde G = \tilde H + \frac{1+x}{2} \tilde G ^ 2 $.
Then we distinguish three cases for a rooted Greg tree:
\begin{itemize}
 \item the root is the vertex with label 1,
 \item the root is unlabeled and has degree 2,
 \item the other cases.
\end{itemize}
Clearly, the first two cases give the generating functions $\tilde H$ and $\frac{x}{2} \tilde G ^2$.
So it remains to check that the other cases give the generating function $\frac{1}{2} \tilde G ^2$.
So let $T\in\mathcal{G}^{\bullet}_n$ be among these other cases. Since the vertex with label 1 is not
the root, we can define an edge $e$ as the first one in the shortest path from the root to the vertex
with label 1. Then, remove this edge and say that its two endpoints are the roots of the two subtrees
thus created. Since the root is not an unlabeled vertex with degree 2, these two subtrees are indeed  
Greg trees. In this way we obtain the generating function $\frac{1}{2} \tilde G ^2$.

In the other direction, we can have $G_n$ in terms of $H_n$. Indeed 
we have the relation:
\begin{equation}
  G_n(x) = (n+(n-1)x)H_n(x) + (x+x^2)H'_n(x).
\end{equation}

This also can be proved combinatorially.

We distinguish the following cases for a rooted Greg tree:
\begin{itemize}
 \item The root is a labeled vertex.
 \item The root is an unlabeled vertex with degree at least 3.
 \item The root is an unlabeled vertex with degree equal to 2.
\end{itemize}
The first two cases give the generating functions $nH_n(x)$ and $xH_n'(x)$. It remains to show that the 
last case give the generating function $(n-1)xH_n(x)+x^2H_n'(x)$. This can be interpreted by the fact 
that the generating function for Greg trees with a distinguished edge (with no restriction) is 
$(n-1)H_n(x)+xH_n'(x)$. We leave the details as an exercise.

\section{The W function has the Bernstein property}

We refer to the book \cite{schilling} for the theory of completely monotonic functions and Bernstein functions.
In particular, we do not present here the motivation for studying these kind of functions. Let us just mention
that Bernstein functions appear in measure theory, because they are related with convolution semigroups of 
probability laws over positive reals. See \cite{pakes} for some probabilistic aspects of the function $W(z)$
related with the Bernstein property.

\setlength{\itemindent}{10mm}
\begin{defi} Let $f : (0,\infty) \to  \mathbb{R}$ of class $C^{\infty}$. We say that
\begin{list}{\labelitemi}{\leftmargin=1em}
 \item $f$ is {\it completely monotonic} if for all $z>0$ and $n\geq0$, $(-1)^{n}f^{(n)}(z)\geq 0$,
 \item $f$ is a {\it Bernstein function} if $f(z)>0$ for any $z>0$, and $f'$ is completely monotonic. 
\end{list}
\end{defi}

\begin{theo}[Kalugin and Jeffrey \cite{kalugin}]
$W(z)$ is a Bernstein function.
\end{theo}

\begin{proof}
The regularity and positivity are rather elementary and the only difficulty is to check that
$(-1)^{n-1}W^{(n)}(z)\geq 0$ for $n\geq 1$. 
From $T(z)=-W(-z)$ and the definition of $G_n$ in \eqref{defG}, we have
\[
(-1)^{n-1} \diff[n]{}{z} W(z) = \frac{  \exp{ (-nW(z)) }  }{  (1+W(z))^n }  G_n\left(   \frac{-W(z)}{1+W(z)}   \right).
\]
Since $W(z)>0$, we have $-1< \frac{-W(z)}{1+W(z)}<0$.
So it remains to show that $G_n(x)>0$ for any $n\geq1$ and $-1<x<0$. To this end, consider
the shifted polynomials $\tilde G_n(x) = G_n(x-1)$. From the recursion satisfied by $G_n(x)$,
we have $\tilde G_1(x)=1$ and 
\[
   \tilde G_{n+1}(x) = n(1+x) \tilde G_n(x) + x^2 \tilde G'_n(x).
\]
By induction, $\tilde G_n(x)$ is seen to have nonnegative coefficients, hence $\tilde G_n(x) \geq 0$ for $0<x<1$,
hence $G_n(x) \geq 0$ for $-1<x<0$. This completes the proof.
\end{proof}

The proof of Kalugin and Jeffrey consists in showing that the polynomials $(-1)^{n-1}P_n$ have positive 
coefficients, which was Sokal's conjecture (they even prove more: the coefficients form a unimodal sequence). 
This positivity of $(-1)^{n-1}P_n$ also follows from properties 
of $G_n(x)$, since the two sequences of polynomials are closely related. Indeed, by comparing 
Equations~\eqref{defP} and \eqref{defG}, we get:
\[
  P_n(x) = (-1-x)^{n-1} G_n\bigg( \frac{-x}{1+x} \bigg).
\]
After the substitution $x \to x-1$, this become
\[
  P_n(x-1) = (-x)^{n-1} G_n\big( \tfrac{1}{x}-1 \big),
\]
which means that the polynomials $G_n(x-1)$ and $ (-1)^{n-1}P_n(x-1)$
are reciprocal of each other.

So we have proved that $W(z)$ is a Bernstein function from Equation~\eqref{defG} and the polynomials $G_n$. 
There are similar results related with the polynomials $F_n$ and $H_n$.

\begin{theo}
 $-T_2(-z)  =  \frac 12 W(z)^2 + W(z) $ is a Bernstein function.
\end{theo}

\begin{proof}
Whenever $f(z)$ is a Bernstein function, $e^{-f(z)}$ is completely monotonic,
see \cite[Theorem~3.6]{schilling}.
In the case of $W(z)$, we have $e^{-W(z)}=\frac{1}{z}W(z)$.
A primitive of $W(z)/z$ is 
\[
  \int \frac{W(z)}{z}  {\rm d}z  =  \frac 12 W(z)^2 + W(z) = -T_2(-z).
\]
Since it has nonnegative values, it means that $\frac 12 W(z)^2 + W(z)$ is also a Bernstein function.

This result can be also proved by checking the signs of the derivatives, and the fact that the 
polynomials $H_n(x)$ defined in Equation~\eqref{defH} are nonnegative on $(-1,0)$.
\end{proof}

\begin{theo}
 $1-T_0(-z) = W(z)/(1+W(z))$ is a Bernstein function.
\end{theo}

\begin{proof}
 A direct calculation shows that $z\mapsto \frac{z}{1+z} $ is a Bernstein function.
 Since the composition of two Bernstein functions has the same propery (see \cite[Corollary 3.7]{schilling}),
 we get the result.

 This can also be proved by checking the sign of the derivatives, and the fact that the polynomials
 $F_n(x)$ defined in Equation~\eqref{defF} are nonnegative on $(-1,0)$.
\end{proof}


In fact, on the analytical level it is rather elementary to prove that $W(z)$ has stronger a property than
being a Bernstein function. We follow the terminology from \cite{schilling}, and 
call Nevanlinna-Pick function an holomorphic function that preserves the complex upper half-plane
$\mathbb{H}^+ = \{z\in\mathbb{C} \, : \, \Im(z)>0 \}$.
Those functions that are nonnegative on $(0,\infty)$ are a particular class of Bernstein, called {\it complete 
Bernstein functions}. See \cite[Chapter 6]{schilling} for details, in particular Theorems~6.2 and 6.7.

\begin{defi}
A Bernstein function $f(z)$ is called {\it complete} if it has an analytic continuation on 
$\mathbb{H}^+$ such that $ f(\mathbb{H}^+) \subset \mathbb{H}^+ $.
\end{defi}

\begin{theo}
$W(z)$ is a complete Bernstein function.
\end{theo}

\begin{proof}
The analytic continuations of $W(z)$ were studied in \cite[Section 4]{corless}. It is possible to extend $W(z)$
as an holomorphic function on $\mathbb{C} \backslash (-\infty,-\frac 1e ] $.
From $W(z)e^{W(z)}=z$, it follows that $W(z)\in\mathbb{R} \Rightarrow z\in\mathbb{R}$.
The image of $\mathbb{H}^+$ is a connected set included in $\mathbb{C}-\mathbb{R}$, so it is
included either in $\mathbb{H}^+$ or $\mathbb{H}^-$.
In a neighborhood of $0$, we have $W(z)=z+O(z^2)$, so we can conclude that $W$ maps
$\mathbb{H}^+$ into itself.
Then, Theorem~6.2 from \cite{schilling} shows the result.
\end{proof}

\section{Improper edges in Cayley trees}

The notion of {\it improper edges} in Cayley trees was introduced by Shor \cite{shor}, and then 
used by Chen and Guo \cite{chen}, Guo and Zeng \cite{guo}, Zeng \cite{zeng}.

\begin{defi}
Let $T \in \mathcal{C}^{\bullet}_n $. A vertex $v_2$ is called a {\it descendant} of a vertex $v_1$ if
the shortest path from the root to $v_2$ goes through $v_1$. Let $e$ be an edge of $T$, and let $u$ and $v$ 
denote its two endpoints in such a way that $v$ is a descendant of $u$. The edge $e$ of $T$ is called 
{\it improper} if the label of $u$ is greater than that of (at least) one of its descendants.
Let $\imp(T)$ denote the number of improper edges of $T$.
\end{defi}

Then Shor's result is:
\[
  G_n(x-1) = \sum_{T\in\mathcal{C}^{\bullet}_n} x^{\imp(T)}.
\]
His definition of this polynomial is a recursion for the coefficients which is equivalent to 
\eqref{recGH}. The unrooted analog (and in fact a more general statement) is due to Zeng \cite{zeng}:
by considering an unrooted tree as a rooted tree where the root is the vertex labeled 1, we have
a notion of improper edge on unrooted tree, and 
\[
  H_n(x-1) = \sum_{T\in\mathcal{C}_n} x^{\imp(T)}.
\]
Hence we have:
\[
          \sum_{T\in\mathcal{C}^{\bullet}_n} (1+x)^{\imp(T)} = \sum_{T\in\mathcal{G}^{\bullet}_n }  x^{\unl(T)},
   \qquad \sum_{T\in\mathcal{C}_n} (1+x)^{\imp(T)} = \sum_{T\in\mathcal{G}_n }  x^{\unl(T)}.
\]
This calls for bijective proofs: we would like to find an explicit map 
$\alpha : \mathcal{G}^{\bullet}_n \to \mathcal{C}^{\bullet}_n$ with the property that, for all
$T\in\mathcal{C}^{\bullet}_n$,
\[
   \sum_{ U \in  \alpha^{-1}(T) }  x^{\unl(U)} = (1+x)^{\imp(T)},
\]
and with the analog property for the unrooted case. We leave this as an open problem for interested readers.

In one direction, we have the map $T\mapsto T|_n$ from Cayley trees to Greg trees (with fewer vertices), and
we have the map $\alpha$ from Greg trees to Cayley trees. It could be interested to see how these map are related.
For example, we can consider the composition $\beta_n: T \mapsto \alpha(T|_n) $, which gives a 
Cayley tree on $n$ vertices for each Cayley tree $T$ on more vertices. The properties of the 
map $\alpha$ and $T\mapsto T|_n$ show that each class $\beta_n^{-1}(T)$ has 
a simple generating function:
\begin{align*}
  \sum_{ m\geq n } \sum_{\substack{ U\in\mathcal{C}_m \\ \beta_n(U)=T }}  \frac{z^{m-n}}{(m-n)!}
     &= \sum_{V\in\alpha^{-1}(T)} \sum_{m\geq n } \sum_{\substack{ U\in\mathcal{C}_m \\ U|_n=V } } \frac{z^{m-n}}{(m-n)!} \\
     &= \sum_{V\in\alpha^{-1}(T)} \frac{e^{nT(z)}}{(1-T(z))^{n-1}} \bigg( \frac{T(z)}{1-T(z)} \bigg)^{\unl(V)} \\
     &= \frac{ e^{nT(z)}  }{ (1-T(z))^{n-1+\imp(T)} }.
\end{align*}
In particular, this property of the map $\beta_n$ gives an alternative way to compute $\diff[n]{}{z} T(z) $, since we get:
\begin{align*}
   \diff[n]{}{z} T(z) &= \sum_{ m\geq n } \sum_{T\in\mathcal{C}_m} \frac{z^{m-n}}{(m-n)!} 
                       = \sum_{ U \in \mathcal{C}_n } \sum_{ T \in \beta_n^{-1}(U) } \frac{ z^{|T|} }{ |T|! } \\
                      &=  \frac{e^{nT(z)}}{(1-T(z))^{n-1}} \sum_{ U \in \mathcal{C}_n } \frac{ 1  }{ (1-T(z))^{\imp(T)} }.
\end{align*}
It could be quite interesting to find an explicit description of such a map $\beta_n$,
although it is a weaker problem than finding one for the map $\alpha$.

\section{Final remarks}
\label{fr}

The three sequence of polynomials studied in this article have a common generalization.
Following Zeng \cite{zeng}, we consider the double sequence of polynomials 
\[
 Q_{n,k} (x) = (x + n - 1)Q_{n-1, k} (x) + (n + k - 2)Q_{n-1, k-1} (x),
\]
where $n\geq 1$, $0\leq k \leq n-1$, and $Q_{1,0}(x)=1$. 
We have $Q_{n,k}(x) = \psi_{k+1}(n-1,x+n)$ where $\psi_k(r,x)$ was defined by Ramanujan \cite{ramanujan}
via the equation
\[
  \sum_{k\geq 0} \frac{ (x+k)^{r+k} e^{-u(x+k)} u^k } {k!} = \sum_{k=1}^{r+1} \frac{ \psi_k(r,x) }{ (1-u)^{r+k} }.
\]
Then one can check that 
\begin{align*}
  \sum_{k=0}^{n-1} Q_{n,k}(-1) x^k &= xF_{n-1}(x-1), \\
  \sum_{k=0}^{n-1} Q_{n,k}(0) x^k  &=  G_n(x-1), \\
  \sum_{k=0}^{n-1} Q_{n,k}(1) x^k  &=  H_n(x-1).
\end{align*}
However, it is not clear if Greg trees are in some way related with $Q_{n,k}(x)$ for 
other values of $x$.

\section{Acknowledgement}

We thank Sloane's OEIS \cite{sloane} for bringing the references \cite{felsenstein,flight} to our knownledge.


\bigskip

\setlength{\parindent}{0pt}

\end{document}